\documentclass[12pt]{amsart}
\usepackage{amsmath,amssymb,latexsym, amsthm, amscd, mathrsfs, stmaryrd, mathtools}
\usepackage[linktocpage=true]{hyperref}
\usepackage{cleveref}
\usepackage{dynkin-diagrams}
\usepackage{xcolor}
\usepackage{graphicx}
\usepackage[all]{xy}
\usepackage{enumitem}
\usepackage{chngcntr}
\usepackage{nicematrix}
\usepackage{multirow}
\usepackage{array}
\usepackage{booktabs}
\usepackage{tikz}
\usepackage{tikz-cd}
\usetikzlibrary{arrows.meta}
\usepackage{bbm}
\usepackage{longtable}
\usepackage{hyperref}
\usepackage{thmtools}
\usepackage[sorting=nyt,style=alphabetic,backend=bibtex,hyperref=true,doi=false,maxbibnames=9,maxcitenames=4,url=false,maxcitenames=50,maxbibnames=50]{biblatex}

\crefformat{section}{\S#2#1#3} 
\crefformat{subsection}{\S#2#1#3}
\crefformat{subsubsection}{\S#2#1#3}
\renewbibmacro{in:}{}
\newtheorem{thm}{Theorem}[section]

\theoremstyle{definition}

\newcommand{\bal}{{\mbox{\boldmath$\alpha$}}}
\newtheorem{Def}[thm]{Definition}
\newtheorem{example}[thm]{Example}
\newtheorem{rem}[thm]{Remark}
\theoremstyle{plain}
\newtheorem{prop}[thm]{Proposition}

\newtheorem{lem}[thm]{Lemma}

\newtheorem{conj}[thm]{Conjecture}

\numberwithin{equation}{section}

\newcommand{\Hom}{\mathrm{Hom}}

\newcommand{\g}{\mathfrak{g}}

\newcommand{\gl}{\mathfrak{gl}}

\newcommand{\KK}{\mathbb K}

\newcommand{\ov}{\overline}

\newcommand{\Z}{\mathbb Z}

\makeatletter
\newcommand{\extp}{\@ifnextchar^\@extp{\@extp^{\,}}}
\def\@extp^#1{\mathop{\bigwedge\nolimits^{\!#1}}}
\makeatother

\DeclareMathOperator*{\bigboxtimes}{\raisebox{-0.25ex}{\scalebox{1.5}{$\boxtimes$}}}
\makeatletter
\newcommand*{\@rowstyle}{}
\newcommand*{\rowstyle}[1]{
  \gdef\@rowstyle{#1}%
  \@rowstyle\ignorespaces%
}
\newcolumntype{=}{
  >{\gdef\@rowstyle{}}%
}
\newcolumntype{+}{
  >{\@rowstyle}%
}
\makeatother

\newcommand{\un}{\mathbbm{1}}

\DeclareMathOperator*{\ad}{ad}

\DeclareMathOperator*{\Rep}{Rep}
\DeclareMathOperator*{\Vecc}{Vec}
\DeclareMathOperator*{\sVec}{sVec}
\DeclareMathOperator{\Ver}{Ver}

\DeclareMathOperator*{\Tilt}{Tilt}
\DeclareMathOperator*{\Lie}{Lie}
\DeclareMathOperator*{\Dist}{Dist}
\DeclareMathOperator*{\Aut}{Aut}
\DeclareMathOperator*{\Fr}{Fr}
\DeclareMathOperator*{\Irr}{Irr}

\pgfdeclarearrow{
name = jibladze,
parameters = { \the\pgfarrowlength },
setup code = {},
drawing code = {
  \pgfsetdash{}{0pt} 
      \pgfsetroundjoin   
  \pgfsetroundcap    
  \pgfsetlinewidth{5\pgflinewidth}
  \pgfsetstrokecolor{white}
  \pgfpathmoveto{\pgfpoint{-.75\pgfarrowlength}{\pgfarrowlength}}
  \pgfpathlineto{\pgfpoint{0}{0}}
  \pgfpathlineto{\pgfpoint{-.75\pgfarrowlength}{-\pgfarrowlength}}
  \pgfusepathqstroke
  \pgfsetlinewidth{.2\pgflinewidth}
  \pgfsetstrokecolor{black}
  \pgfpathmoveto{\pgfpoint{-.75\pgfarrowlength}{\pgfarrowlength}}
  \pgfpathlineto{\pgfpoint{0}{0}}
  \pgfpathlineto{\pgfpoint{-.75\pgfarrowlength}{-\pgfarrowlength}}
  \pgfusepathqstroke
},
defaults = { length = 2pt }
}

\title[The Steinberg Tensor Product Theorem for $GL(X)$ in $\Ver_p$]{The Steinberg Tensor Product Theorem for General Linear Group Schemes in the Verlinde Category}

\author[A.S. Kannan]{Arun S. Kannan}%
\address{Department of Mathematics, Massachusetts Institute of Technology, Cambridge, MA 02139} \email{akannan@mit.edu}

\begin{document}

\begin{abstract}
The Steinberg tensor product theorem is a fundamental result in the modular representation theory of reductive algebraic groups. It describes any finite-dimensional simple module of highest weight $\lambda$ over such a group as the tensor product of Frobenius twists of simple modules with highest weights the weights appearing in a $p$-adic decomposition of $\lambda$, thereby reducing the character problem to a a finite collection of weights. In recent years this theorem has been extended to various quasi-reductive supergroup schemes. In this paper, we prove the analogous result for the general linear group scheme $GL(X)$ for any object $X$ in the Verlinde category $\Ver_p$. 
\end{abstract}

\maketitle

\setcounter{tocdepth}{1}
\tableofcontents

\section{Introduction}
\subsection{The Steinberg tensor product theorem for quasi-reductive supergroups}
Let $\KK$ be an algebraically closed field of characteristic $p > 0$, and let $G$ be a reductive algebraic group over $\KK$ with maximal torus $T$ and $r$-th Frobenius kernel $G_{r}$. Recall that the set of dominant integral weights $\Lambda(T)_+$ for $G$ indexes the finite-dimensional simple modules over $G$ (up to isomorphism), which we will denote $L(\lambda)$ for $\lambda \in \Lambda(T)_+$. Moreover, $\Lambda(T)_+$ admits a filtration by finite sets $\Lambda_1(T) \subset \Lambda_2(T) \subset \cdots \subset \Lambda(T)_+$, where $\Lambda_r(T)$ indexes the set of finite-dimensional simple $G_r$-modules for each $r \geq 1$ (the modules themselves are given by restricting the corresponding $G$-module to $G_r$). The \textit{Steinberg tensor product theorem} states that for any $\lambda \in \Lambda_r(T)$ and $\mu \in \Lambda(T)_+$ we have

\[L(\lambda + p^r \mu) \cong L(\lambda) \otimes L(\mu)^{[r]},\]
where the operation $[r]$ on a $G$-module denotes the $r$-th Frobenius twist of that module. By induction we can show that if $\lambda$ admits a $p$-adic decomposition $\lambda = \sum_{i=0}^r p^i \lambda_i$ with $\lambda_i \in \Lambda_1(T)$ for all $i$, then

\[L(\lambda) \cong L(\lambda_0)\otimes L(\lambda_1)^{[1]} \otimes \cdots \otimes L(\lambda_r)^{[r]},\]
which is how Steinberg originally phrased it (see \cite{steinberg1963} or \cite[\S II.3.17]{jantzen2003representations}). The most important point of this theorem is that it reduces the unsolved problem of computing the character of $L(\lambda)$ for $\lambda \in \Lambda(T)_+$ to determining that of $L(\mu)$ for $\mu \in \Lambda_1(T)$.
\par
Similar results hold for various \textit{quasi-reductive} supergroup schemes, meaning the underlying ordinary (more traditionally referred to as \textit{even}) group $G_0$ is a reductive group. In particular, analogs of the Steinberg tensor product theorem were proven for $GL(m|n)$, $Q(n)$, and $OSp(m|2n)$ (see \cite{kujawa2006steinberg,BRUNDAN_Qn,binwang_orthosymplectic}, respectively), and later generalized to a broader a class of quasi-reductive supergroups in \cite{shibata_frobenius_2023}. Let us state the theorem for the general linear supergroup $G = GL(m|n)$. The Frobenius morphism will send $G$ to the underlying ordinary subgroup $G_0 = GL(m) \times GL(n)$, meaning any representation $V$ of the underlying ordinary group pulls back to a representation $F^*V$ of $G$ via the Frobenius morphism.  Moreover, finite-dimensional simple $G$-modules $\{L(\lambda)\}$ and finite-dimensional simple $G_0$-modules $\{L_0(\lambda)\}$ are indexed by the same set of weights $\Lambda(T)_+$ (easily seen by parabolic induction). The Steinberg tensor product theorem for $GL(m|n)$ states that for any $\lambda \in \Lambda_1(T)$ and $\mu \in \Lambda(T)_+$ we have 

\[L(\lambda + p\mu) \cong L(\lambda)\otimes F^* L_0(\mu).\]

\subsection{Generalizing to the Verlinde category}
Symmetric tensor categories arise as an axiomatization of the properties of representation categories of finite groups (see \cite{etingof2016tensor}) and can be thought of as a ``home'' to do commutative algebra, algebraic geometry, and Lie theory. Deligne's theorem states (cf. \cite{deligne2002categories}) that over algebraically closed fields of characteristic zero, all symmetric tensor categories of \textit{moderate growth}, meaning lengths of tensor powers grow subexponentially, are representation categories of affine supergroup schemes, roughly speaking. In other words, such categories \textit{fiber} over the symmetric tensor category $\sVec_\KK$ of supervector spaces. In positive characteristic, Deligne's theorem is not true; a simple but fundamental counterexample for $p > 3$ is the \textit{Verlinde category} $\Ver_p$ (when $p = 2,3$ it is just the category of vector spaces and supervector spaces, respectively). 
\par
The symmetric tensor category $\Ver_p$ is defined as the \textit{semisimplification} of the representation category $\Rep ((\mathbb{G}_a)_{(1)})$ of the first Frobenius kernel of the additive group scheme $\mathbb{G}_a$ (see \cite{ostrik2020symmetric,etingof2021semisimplification}). Informally speaking, this can be thought of as the category tensor generated by declaring indecomposable objects to be simple and sending indecomposable objects with categorical dimension zero to the zero object. In summary $\Ver_p$ is additively generated by $p-1$ simple objects $L_1, L_2, \dots, L_{p-1}$, subject to a Clebsch-Gordan-esque fusion rule. Notably, $L_1$ and $L_{p-1}$ generate $\sVec_\KK$ (for $p > 2$). More generally for any reductive group $G$ with Coxeter number $h < p$ we can define a category $\Ver_p(G)$, which is the semisimplification of the category of tilting modules over $G$ (when $G = SL(2)$ we recover $\Ver_p$). These contain certain symmetric tensor subcategories called $\Ver_p^+(G)$.
\par
The category $\Ver_p$ plays an important role in the quest for generalizing Deligne's theorem to positive characteristic. In \cite{coulembier2022frobenius}, it is shown that a symmetric tensor category fibers over $\Ver_p$ if and only if it is \textit{Frobenius-exact}, a class of symmetric tensor categories that includes all semisimple symmetric tensor categories. Therefore, it is worthwhile to study the representation theory of affine group schemes in $\Ver_p$, the most fundamental being that of the group scheme $GL(X)$ for any object $X$ in $\Ver_p$. Moreover, our answers must generalize what we know about $GL(m|n)$, as $\sVec_\KK \subseteq \Ver_p$ for $p > 2$.
\par
In \cite{venkatesh_glx_2022}, it is shown that finite-length (as objects in $\Ver_p$) simple modules over $G = GL(X)$ for an object $X = \bigoplus_{k=1}^{p-1}L_k^{\oplus n_k}$ in $\Ver_p$ are indexed by pairs $(\lambda, V)$, where $\lambda$ is a dominant integral weight for $GL(X)_0$ and $V \in {\bigboxtimes}_{k=1}^{p-1}\Ver_p^+(SL(k))^{\boxtimes n_k}$ is simple. We will denote the corresponding simple $GL(X)$-module as $L(\lambda, V)$.
\par
Similar to the ordinary vector space setting, one can define the Frobenius morphism on $GL(X)$ (see \cite{coulembier2022frobenius}). In particular, the image of $GL(X)$ under the Frobenius morphism lies in $G_0$, which lets us pullback representations from $G_0$ to $G$. We are now ready to state the main theorem of this paper:

\begin{restatable}{theorem}{maintheorem}
\label{thm:main_theorem}
    Let $X = \bigoplus_{k=1}^{p-1}L_k^{\oplus n_k}$ be an object in $\Ver_p$ and $G = GL(X)$ with $G_0 = \prod_{k=1}^n GL(n_k)$. Moreover, let $\Lambda(T)_+$ denote the set of dominant integral weights for $G_0$ and $\Lambda_1(T) \subset \Lambda(T)_+$ the weights arising from the first Frobenius kernel of $G_0$. Then, for any $\lambda \in \Lambda_1(T)$, $\mu \in \Lambda(T)_+$ and $V$ a simple object in ${\bigboxtimes}_{k=1}^{p-1}\Ver_p^+(SL(k))^{\boxtimes n_k}$, we have an isomorphism of $G$-modules:

    \[L(\lambda + p\mu, V) \cong L(\lambda, V)\otimes \Fr\nolimits_{+,\uparrow}(L_0(\mu)),\]
    where $L_0(\mu)$ is the simple module over $G_0$ of highest weight $\mu$, and $\Fr_{+,\uparrow}$ is the inflation induced by the Frobenius morphism $\Fr_+: G \rightarrow G_0$.
\end{restatable}
In particular, this generalizes the main theorem of Kujawa in \cite{kujawa2006steinberg} to the Verlinde category $\Ver_p$. The proof is essentially identical to that in \textit{loc. cit.} but has been adapted for the setting of $\Ver_p$.

\subsection{Outline of paper}
In \S\ref{sec:prelim}, we review some basic properties about the Verlinde categories of reductive groups. We then recall basic properties about affine group schemes in symmetric tensor categories as well as about their representation theory. We conclude the section by recalling how the affine group scheme $GL(X)$ is constructed and how its simple modules are constructed and classified.
\par
In \S\ref{sec:frob_reps}, we state the definition of the Frobenius twist and the Frobenius morphism on an affine group scheme in a symmetric tensor category that admits a symmetric tensor functor to $\Ver_p$. We then define Frobenius kernels and review their basic properties. We finish by constructing the simple modules for the Frobenius kernels of $GL(X)$. 
\par
Finally, in \S\ref{sec:proof}, we prove Theorem \ref{thm:main_theorem}.
\subsection{Acknowledgements}
I would like to thank my Ph.D. advisor Pavel Etingof for his guidance and support and the anonymous referee for proofreading and helpful comments.

\section{Preliminaries}\label{sec:prelim}
\subsection{Notation and conventions} For the rest of this paper, let us assume that $\KK$ is an algebraically closed field of characteristic $p > 3$ and all symmetric categories are defined over $\KK$ \footnote{We can also take $p = 2,3$ in most instances, but in this case the content of the main theorem is nothing new.}. We assume that the reader is familiar with symmetric tensor categories (see \cite{etingof2016tensor,etingof2021lectures}) and with reductive algebraic groups and their representation theory (see \cite{jantzen2003representations}). We will let $\un$ denote the unit object of a symmetric tensor category, $V^*$ denote the dual of an object $V$, $\boxtimes$ denote the Deligne tensor product of two symmetric tensor categories, and if $G$ is a group, let $\Rep G$ denote the representation category of $G$. We agree that all symmetric tensor functors out of a tensor category are exact. Let $\mathcal{C}^{ind}$ denote the ind-completion of a symmetric tensor category; in the case $\mathcal{C}$ is semisimple, this is just the category whose objects are potentially infinite direct sums of objects in $\mathcal{C}$. An object in a symmetric tensor category $\mathcal{C}$ fibered over the Verlinde category $\Ver_p$ will be called \textit{finite-dimensional} if its image under the fiber functor is the direct sum of finitely many simple objects in $\Ver_p$. Finally, given an object $A$ in a semisimple tensor category, let $A_0$ denote the subobject of $A$ corresponding to the isotypic component of the unit object, and let $A_{\neq 0}$ denote the subobject of $A$ corresponding to the sum of all other isotypic components.
\subsection{Verlinde categories}
To construct the Verlinde categories $\Ver_p(G)$ associated to a reductive algebraic group $G$ with Coxeter number $h < p$, we will first briefly review properties of a procedure known as \textit{semisimplification} and of tilting modules over $G$.
\par
If $\mathcal{C}$ is a symmetric tensor category (or, more generally, a symmetric \textit{pseudotensor} category), then its \textit{semisimplification} $\ov{\mathcal{C}}$ is, loosely speaking, the symmetric tensor category obtained by declaring indecomposable objects in $\mathcal{C}$ to formally be simple and sending indecomposable objects with categorical dimension zero to the zero object. There always exists a symmetric, monoidal, additive, and $\KK$-linear functor $\mathcal{C} \to \ov{\mathcal{C}}$ called the \textit{semisimplification functor}, as $\ov{\mathcal{C}}$ is the quotient of $\mathcal{C}$ by the tensor ideal of \textit{negligible morphisms}. We will denote the image of an object under the semisimplification functor with an overline over the object. See \cite{etingof2021semisimplification} for more details.
\par
A tilting module over $G$ is a module over $G$ which admits two filtrations, one whose successive subquotients are standard modules and the other whose successive subquotients are costandard modules. It is well known that for each $\lambda \in \Lambda(T)_+$, the set of dominant integral weights, there is an indecomposable tilting module $T(\lambda)$ with highest weight $\lambda$ (see \cite{jantzen2003representations}). These coincide with the simple $G$-module $L(\lambda)$ when $\lambda \in C^0$, the fundamental alcove. For $\lambda \in \Lambda(T)_+ \backslash C^0$, the dimension of $T(\lambda)$ is divisible by $p$ i.e. it is of categorical dimension zero. Moreover, any tilting module is the direct sum of some multiplicity of $T(\lambda)$ for various $\lambda \in \Lambda(T)_+$, which implies that there is a fusion rule for the tensor product $T(\lambda)\otimes T(\mu) = \bigoplus_{\nu \in \Lambda(T)_+} T(\nu)^{\oplus a_{\lambda\mu}^{\nu}}$. The upshot is that the full subcategory of tilting modules $\Tilt G$ of $\Rep G$ is a symmetric pseudotensor category, and we have the following definition:

\begin{Def}\cite{georgiev1994fusion, gelfand1992examples}
    Given a reductive algebraic group $G$ whose Coxeter number $h$ is less than $p$, we define the symmetric tensor category $\Ver_p(G)$ to be the semisimplification of the category $\Tilt G$ of tilting modules over $G$. The simple objects of $\Ver_p(G)$ are indexed by the dominant integral weights in the fundamental alcove $C^0$ of $G$. 
\end{Def}
In the case $G = SL(n)$, whose Coxeter number is $n$, we can define the category $\Ver_p(G)$ when $p > n$. The category $\Ver_p(SL(n))$ always admits a decomposition $\Ver_p(SL(n)) = \Ver_p^+(SL(n)) \boxtimes \mathcal{P}$, where $\Ver_p^+(SL(n))$ is a suitable symmetric tensor subcategory and $\mathcal{P}$ is the \textit{pointed} subcategory of $\Ver_p(G)$ generated by the \textit{invertible objects} i.e. simple objects $X$ such that $X \otimes X^* \cong \un$ (see \cite[Proposition 4.5]{venkatesh_glx_2022}). By studying the fusion rules of $\Ver_p(SL(n))$, we can deduce the following properties:

\begin{enumerate}
    \item Simple objects in $\Ver_p(SL(n))$  are indexed by partitions whose first part is at most $p-n$ and which have at most $n-1$ parts. These correspond precisely to weights in the fundamental alcove in the obvious way and we will therefore use the fundamental alcove to index these.
    \item  If $\lambda, \mu \in C^0$ are such that $n$ divides both $|\lambda|$ and $|\mu|$ (their sizes), then in $\Ver_p(G)$ we have $\ov{T(\lambda)} \otimes \ov{T(\mu)} = \bigoplus_{\nu \in C^0} \ov{T(\nu)}^{\oplus a_{\mu\lambda}^{\nu}}$, where $a_{\mu\lambda}^{\nu} = 0$ unless $n$ also divides $|\nu|$. Therefore, it can be shown that the simple objects in $\Ver_p^+(SL(n))$ are $\ov{T(\lambda)}$ for $\lambda \in C^0$ such that $n$ divides $|\lambda|$.
    \item The invertible objects in $\Ver_p(SL(n))$ (i.e. simple objects in $\mathcal{P}$) correspond to partitions of the form $((p-n)^k)$ for $0 \leq k \leq n-1$, with the unit object corresponding to the empty partition. Their multplication rule is given by $\ov{T((p-n)^k)} \otimes \ov{T((p-n)^j)} = \ov{T((p-n)^{(j+k) \mod n})}$. More generally, if $\lambda \in C^0$, then $\ov{T((p-n))} \otimes \ov{T(\lambda)} = \ov{T(\mu)}$ where $\mu$ is given as follows: visualize the two partitions $(p-n)$ and $\lambda$ as Young diagrams, stack $(p-n)$ on top of $\lambda$, and then delete the columns left to right until there are at most $n-1$ parts.
\end{enumerate}
The case $n = 2$ is particularly important for us:

\begin{Def}
    We denote $\Ver_p \coloneqq \Ver_p(SL(2)) = \Ver_p(SL(p-2))$, and we simply call it the \textit{Verlinde category}.
\end{Def}
The symmetric tensor category $\Ver_p$ has $p-1$ simple objects $L_1, L_2, \dots, L_{p-1}$, which are self-dual, and their fusion rule is given by the truncated Clebsch-Gordan rule:

\begin{equation*}\label{clebschgordan}
    L_m \otimes L_n = \bigoplus_{i=1}^{\min(m,n,p-m,p-n)} L_{|m-n| + 2i - 1}.
\end{equation*}
The invertible objects are $L_1$ and $L_{p-1}$, and in this case $\mathcal{P} = \sVec_\KK$ (see \cite{ostrik2020symmetric}). We let $\Ver_p^+ \coloneqq \Ver_p^+(SL_2)$, which is spanned by the objects $L_i$ for $i$ odd, and we have $\Ver_p = \Ver_p^+ \boxtimes \sVec_\KK$. The Verlinde category is proving to be a central object in the study of symmetric tensor categories in positive characteristic (see \cite{ostrik2020symmetric,coulembier2020tannakian, benson2023new, coulembier2022frobenius,coulembier2023incompressible}). For instance, there is always a forgetful functor $\Ver_p(G) \rightarrow \Ver_p(SL(2))$ given by inclusion of a principal $SL(2)$ in $G$, which at the level of representation categories is given by the restriction functor, which sends negligible morphisms to negligible morphisms (see \cite{ostrik2020symmetric} for details).
\par
The Verlinde category admits another description which will be useful for us. Let $\bal_p \coloneqq (\mathbb{G}_a)_{(1)}$ denote the first Frobenius kernel of the additive group scheme $\mathbb{G}_a$. Because $\mathcal{C} \coloneqq \Rep \bal_p$ can be identified with $\Rep \KK[t]/(t^p)$, by the theory of Jordan canonical forms, the category $\mathcal{C}$ has $p$ indecomposable objects, $J_k$ for $1 \leq k \leq p$, with $\dim J_k = k$. The semisimplification $\ov{\mathcal{C}}$ of $\mathcal{C}$ is $\Ver_p$, with $\ov{J_k} = L_k$ for $1 \leq k \leq p-1$.
\subsection{Affine group schemes in $\Ver_p$ and associated constructions}
For this subsection, let $\mathcal{C}$ denote $\Ver_p$, although most definitions below extend to any semisimple symmetric tensor category. Recall that notions of an algebra, commutative algebra, Lie algebra, or Hopf algebra in $\mathcal{C}$ and morphisms between them can be defined by diagrammatizing the usual definitions (here we assume in the definition that a Lie algebra is an operadic Lie algebra that satisfies the PBW theorem, see \cite{etingof2018koszul} for more details). 

\begin{Def}
An \textit{affine group scheme} $G$ in $\mathcal{C}$ is a representable functor $\mathbf{CommAlg}_{\mathcal{C}} \rightarrow \mathbf{Grp}$ from the category of commutative algebras in $\mathcal{C}$ to the category of groups. The usual argument shows that such a functor is represented by a commutative Hopf algebra $\mathcal{O}(G)$ in $\mathcal{C}^{ind}$, which we call the \textit{coordinate ring of $G$} (i.e. $G(A) = \Hom_{\mathbf{CommAlg}_{\mathcal{C}}}(\mathcal{O}(G), A)$ for any commutative algebra $A \in \mathcal{C}$). If the coordinate ring $\mathcal{O}(G)$ is finitely generated, meaning it is a quotient of the symmetric algebra $S(V)$ for some object $V \in \mathcal{C}$, then we say that $G$ is of \textit{finite type}. If the coordinate ring is finite-dimensional, we say that $G$ is \textit{finite}.
\end{Def}
In the category of vector spaces, various notions in commutative and algebraic geometry relevant to the study of group schemes have been developed over the past several decades; see \cite[\S I]{jantzen2003representations} for a reasonable collection of such definitions and results we might want or expect. However, formulating such notions is in its infancy for symmetric tensor categories in general (see \cite{zubkov2009, MASUOKA2011135, MASUOKA202289} and references therein for $\sVec_\KK$, and \cite{coulembier2023commutative_alg, coulembier2023algebraic_geo} for symmetric tensor categories in general). For instance, it is unclear what the correct definition of a normal subgroup should be. To get around these ambiguities, we will appeal to the theory of \textit{Harish-Chandra pairs}, which was first formulated for supergroups in arbitrary characteristic in \cite{masuoka2012harish} and extended to affine group schemes in the Verlinde category in \cite{venkatesh_hcpairs}. To do so, we need to introduce some notions and in the remainder of the subsection will freely refer to results in \cite{venkatesh_hcpairs}. 

\begin{Def}
The \textit{underlying ordinary subgroup} $G_0$ of an affine group scheme $G$ is the affine group scheme represented by the quotient Hopf algebra $\mathcal{O}(G)/\langle \mathcal{O}(G)_{\neq 0}\rangle$. 
\end{Def}
The group scheme $G_0$ can also be thought of as the restriction of $G$ to the category of ordinary commutative algebras. It is obviously a closed subgroup. Moreover, let $I$ denote the \textit{augmentation ideal} of $\mathcal{O}(G)$, which by definition is the kernel of the counit map on $\mathcal{O}(G)$. We can define the \textit{distribution algebra} of $G$ and its \textit{Lie algebra} $\Lie(G)$ as usual:

\begin{Def}
The \textit{distribution algebra} (or hyperalgebra) $\Dist(G)$ of $G$ is given by 

\[\Dist(G) \coloneqq \bigcup_{n=0}^\infty (\mathcal{O}(G)/I^n)^*,\]
which admits the structure of an algebra by dualizing the comultiplication on $\mathcal{O}(G)$. The subobject of primitives in  $\Dist(G)$, which can be identified with $(I/I^2)^*$, is called the \textit{Lie algebra} $\Lie(G)$ of $G$.
\end{Def}
Because it is a subobject of the associative algebra $\Dist(G)$ closed under bracket given by the categorical formulation of the commutator, $\Lie(G)$ is automatically a Lie algebra, an operadic Lie algebra satisfying the PBW theorem (see \cite{etingof2018koszul} for more details). It can also be checked that $\Dist(G_0)$ is a subalgebra of $\Dist(G)$ and that $\Lie(G_0)$ is $\Lie(G)_0$. We list some properties about distribution algebras and coordinate rings below (in $\Ver_p$):

\begin{enumerate}
    \item let $X \in \Ver_p$ be equipped with trivial bracket and suppose $\Hom(\un, X) = 0$. Then $X$ is the Lie algebra of a group scheme represented by $S(X)^*$, which is necessarily finite-dimensional, as $S(L_i) = 0$ for all $i > 1$ (see \cite[Proposition 2.4]{etingof2017computations});
    \item we have an isomorphism $\mathcal{O}(G) \cong \mathcal \mathcal{O}(G_0) \otimes S(\Lie(G)_{\neq 0}^*)$ as left $\mathcal{O}(G_0)$-comodule algebras (see \cite[Lemma 7.13]{venkatesh_hcpairs});
    \item  dually, as left $\Dist(G_0)$-module coalgebras, we have an isomorphism $\Dist(G) \cong \mathcal \Dist(G_0) \otimes S(\Lie(G)_{\neq 0})$ (see \cite[Theorem 6.3]{venkatesh_hcpairs}).
\end{enumerate}

Finally, we require one more standard but important definition: 
\begin{Def}\label{hcpairs_def}
A \textit{Harish-Chandra pair} is a pair $(H_0, \mathfrak{h})$, where $H_0$ is a ordinary affine group scheme of finite type over $\KK$ and $\mathfrak{h}$ is a Lie algebra in $\mathcal{C}$ such that:
\begin{enumerate}
    \item $\mathfrak{h}_0 = \Lie(H_0)$ is a Lie subalgebra of $\mathfrak{h}$;
    \item $H_0$ acts on $\mathfrak{h}$ by Lie algebra automorphisms and this action restricts to the adjoint action on $\mathfrak{h}_0$;
    \item the differential of the action of $H_0$ on $\mathfrak{h}$ is the adjoint action of $\mathfrak{h}_0$ on $\mathfrak{h}$.
\end{enumerate}

\end{Def}
Collectively, these form a category $\mathbf{HC}$, whose morphisms consist of a group homomorphism and a Lie algebra homomorphism so that the compatibility conditions in Definition \ref{hcpairs_def} are preserved in the obvious way. We can now state the main theorem of \cite{venkatesh_hcpairs}:

\begin{thm}{\cite[Theorem 1.2, Corollary 1.3]{venkatesh_hcpairs}}\label{hcpairs_theorem}
    There is an equivalence of categories between the category of affine group schemes of finite type in $\Ver_p$ and the category $\mathbf{HC}$, where an affine group scheme $G$ of finite type is sent to the pair $(G_0, \Lie(G))$. In particular, the closed subgroups of $G$ are in bijection with Harish-Chandra pairs $(H_0, \mathfrak{h})$, where $H_0$ is a closed subgroup of $G_0$ and $\mathfrak{h}$ is a $H_0$-stable Lie subalgebra of $\Lie(G)$ such that $\Lie(H_0) = \mathfrak{h}_0$.
\end{thm}
In particular, we will call such a subgroup as in Theorem \ref{hcpairs_theorem} \textit{normal} if $H_0$ is a normal subgroup of $G_0$ and $\mathfrak{h}$ is an $H_0$-stable ideal in $\Lie(G)$.

\subsection{Representations of affine group schemes}
The usual definition of a representation of an affine group scheme $G$ using either comodules, a homomorphism into a suitable general linear group (defined in \ref{sec:glx}), or the functor of points definition all extend to symmetric tensor categories as the definitions can be phrased categorically (see \cite[\S I.2]{jantzen2003representations}). We will denote the corresponding category of representations as $\widetilde{\Rep G}$. However, we want to consider a certain natural class of representations, which we will define below.
\par
Given a symmetric tensor category $\mathcal{C}$, we can associate an affine group scheme $\pi(\mathcal{C})$ in $\mathcal{C}$ called its \textit{fundamental group} (the notion is originally due to Deligne, see \cite[\S 2.4]{etingof_complex_rank_2} for more details). In the case $\mathcal{C}$ is semisimple, $\pi(\mathcal{C})$ is represented by $\mathcal{O}(\pi(\mathcal{C})) \coloneqq \bigoplus_{X \in \Irr \mathcal{C}} X \otimes X^*$, where $\Irr \ \mathcal{C}$ denotes the set of isomorphism classes of simple objects in $\mathcal{C}$. Every object in $\mathcal{C}$ has a coaction by $\mathcal{O}(\pi(\mathcal{C}))$, so the fundamental group acts on every object in the category. If $G$ is an affine group scheme in $\mathcal{C}$, then there is a natural forgetful functor $\widetilde{\Rep G} \rightarrow \mathcal{C}$. This in turn yields a group homomorphism $\rho: \pi(\mathcal{C}) \rightarrow G$. We can now make the following definition:
\begin{Def}
Let $G$ be an affine group scheme in $\mathcal{C}$. We let $\Rep G \subseteq \widetilde{\Rep G}$ be the full subcategory of $\widetilde{\Rep G}$ consisting of $G$-modules $M$ such that the action of the fundamental group $\pi(\mathcal{C})$ on $M$ as an object in $\mathcal{C}$ coincides with its action via $\rho: \pi(\mathcal{C}) \rightarrow G$.
\end{Def}
\subsection{The group scheme $GL(X)$ and its representation theory}\label{sec:glx}
We are now ready to describe both the group scheme $GL(X)$ for an object $X \in \Ver_p$ and its representation theory. We freely reference \cite{venkatesh_glx_2022}, where basic properties about $GL(X)$ were first proven and the simple $GL(X)$-modules were first classified. 
\par
For an object $X \in \Ver_p^+$, let $\gl(X) \coloneqq X \otimes X^*$ denote the matrix algebra on $X$ (we will also use this notation to refer to the Lie algebra $\gl(X)$ later). Its multiplication arises from the evaluation map of $X^*$ on $X$ as follows:

\[X \otimes X^* \otimes X \otimes X^* \xrightarrow{1_X \otimes ev_{X^*,X}\otimes 1_{X^*}} X \otimes \un \otimes X^* \cong X \otimes X^*,\]
analogous to usual matrix multiplication. This yields a map

\[m^*: S( \gl(X)^* ) \rightarrow S( \gl(X)^* ) \otimes S( \gl(X)^* )\] 
of commutative algebras. Moreover, the universal property of symmetric algebras applied to the coevaluation map $\un \rightarrow \gl(X)$ yields a map $S(\gl(X)^*) \rightarrow \un$, whose kernel we will call $K$. Then, we have the following definition:

\begin{Def}
    The \textit{general linear group scheme} $GL(X)$ is by definition the affine group scheme represented by the quotient Hopf algebra $S( \gl(X)^* ) \otimes S( \gl(X)^* )  / \langle m^*(K) \rangle$.
\end{Def}
The definition comes from mimicking the usual definition of the general linear group as the scheme defined by the equation $AB = 1$. From this definition the following properties can be shown:

\begin{enumerate}
    \item  An explicit description as a functor of points is that $GL(X)(A) = \Aut_{A}(A\otimes X)$ is the group of $A$-module automorphisms of the $A$-module $A \otimes X$ for $A$ a commutative algebra in $\Ver_p$ (see \cite[Proposition 3.1]{venkatesh_glx_2022}).
    \item If $X = \bigoplus_{k=1}^{p-1}L_k^{\oplus n_k}$, then the underlying ordinary group of $GL(X)$ is given by
    
    \[GL(X)_0 = \prod_{k=1}^{p-1} GL(n_k).\]
    See \cite[Corollary 3.3]{venkatesh_glx_2022}.
    \item $\Lie(G) = \mathfrak{gl}(X)$, where the bracket is given by the categorical formulation of the commutator arising from the matrix multiplication above (see \cite[Corollary 3.3]{venkatesh_glx_2022}). 
    \item Define the special linear Lie algebra $\mathfrak{sl}(X)$ to be the kernel of the map 
    
    \[X\otimes X^* \xrightarrow{c_{X,X^*}} X^* \otimes X \xrightarrow{ev_{X^*,X}} \un,\] 
    where the first arrow is application of the braiding and the second is the evaluation of an object on its dual. Then, if $\dim(X) \neq 0$, we have an isomorphism 
    
    \[\mathfrak{gl}(X) = \mathfrak{sl}(X) \oplus \un.\] 
    Moreover, $\mathfrak{sl}(L_i)$ is a simple Lie algebra with $\Hom(\un ,\mathfrak{sl}(L_i)) = 0$ for all $1 < i < p-1$ (it is zero otherwise) (see \cite[Proposition 3.5, Theorem 4.3]{venkatesh_glx_2022}).
\end{enumerate}
It will be useful to denote various subgroups of $GL(X)$, which we will phrase using the correspondence between subgroups and Harish-Chandra pairs. Let us fix an ordering on the simple objects appearing in $X$, so that simple objects that are isomorphic appear consecutively. For example, if $X = L_1^{\oplus 2} \oplus L_3^{\oplus 3}$, write $X = L_1 \oplus L_1 \oplus L_3 \oplus L_3 \oplus L_3$. With respect to this ordering we write $X = \bigoplus_{k=1}^n X_k$, where $n = n_1 + \cdots + n_{p-1}$ and each $X_k$ is simple. Then, we can think of $\gl(X)$ as a matrix with the object $X_i \otimes X_j^*$ in the $(i,j)$-th entry. We will define the following subgroups of $GL(X)$:

\begin{enumerate}
    \item The standard maximal torus 
    
    \[T(X) \coloneqq \left(T(GL(X)_0), \mathfrak{t}(X)\right),\] where $T(GL(X)_0)$ is the maximal torus of diagonal matrices in the underlying ordinary subgroup $GL(X)_0 = \prod_{k=1}^{p-1} GL(n_k)$ and $\mathfrak{t}(X) = \bigoplus_{1 \leq i \leq n} \gl(X_i)$ is the Lie subalgebra of diagonal matrices in $\mathfrak{gl}(X)$. The representation theory of $T(X)$ is semisimple, and therefore any $GL(X)$-module admits a decomposition as the direct sum of $T(X)$-modules, which we call its \textit{weight-space decomposition} (with respect to $T(X)$). If $L$ is a $GL(X)$-module, then the isotypic component of $V$ in $L$ for an irreducible $T(X)$-module $V$ will be denoted by $L_V$. 
    \item The standard Borel 
    
    \[B(X) \coloneqq \left(\prod_{k=1}^{p-1} B(n_k), \mathfrak{b}(X)\right),\] 
    where $B(n_k)$ is the subgroup of upper-triangular matrices in $GL(n_k)$ and $\mathfrak{b}(X) = \bigoplus_{1 \leq i \leq j \leq n} X_i \otimes X_j^*$ is the Lie subalgebra of upper-triangular matrices in $\mathfrak{gl}(X)$.
    \item The strictly lower-triangular matrices $N^-(X) \coloneqq \left(\prod_{k=1}^{p-1} N^-(n_k), \mathfrak{n}^-(X)\right)$, where $N^-(n_k)$ is the subgroup of strictly lower-triangular matrices in $GL(n_k)$ and 
    
    \[\mathfrak{n}^-(X) = \bigoplus_{1 \leq j < i \leq n} X_i \otimes X_j^*\] 
    is the Lie subalgebra of strictly lower triangular matrices in $\mathfrak{gl}(X)$.
    \item The parabolic subgroup $P(X)$ is defined by
    
    \[P(X) \coloneqq \left(GL(X)_0, \mathfrak{b}(X) + \sum_{i=1}^{p-1}\gl(L_i^{\oplus n_i}) \right).\]
\end{enumerate}
We briefly recall how the simple $GL(X)$-modules are constructed in \cite{venkatesh_glx_2022}, as it will be relevant for how we determine simple modules over Frobenius kernels of $GL(X)$ later in the paper. First, one can deduce that $\Rep GL(L_k) = \Rep \mathbb{G}_m \boxtimes \Ver_p^+(SL(k))$, where $\mathbb{G}_m$ is the multiplicative group. Then, via a standard Verma-module argument using distribution algebras of $GL(L_k^{\oplus n_k})$, one finds that the simple $GL(L_k^{\oplus n_k})$ modules are indexed by pairs consisting of a dominant integral weight for $GL(n_k)$ and a simple object in $\Ver_p^+(SL(k))^{\boxtimes n_k}$. Finally, one uses parabolic induction with respect to the subgroup $P(X)$ to deduce the following theorem:

\begin{thm}{\cite[Theorem 1.3]{venkatesh_glx_2022}}
    Let $X = \bigoplus_{k=1}^{p-1}L_k^{\oplus n_k}$ be an object in $\Ver_p$. Then, the finite-dimensional simple $GL(X)$-modules are indexed by pairs $(\lambda, V)$, where $\lambda$ is a dominant integral weight for $GL(X)_0$ and $V \in {\bigboxtimes}_{k=1}^{p-1}\Ver_p^+(SL(k))^{\boxtimes n_k}$ is simple. 
\end{thm}
We will denote the corresponding simple $GL(X)$-module as $L(\lambda, V)$. It is shown in \textit{loc. cit.} that these are all compatible with the action of the fundamental group of $\Ver_p$.

\subsection{Weights and Roots for $GL(X)$}
Finally, let us conclude this subsection by setting up notation to describe a first-pass at what weights and roots for $GL(X)$ might be. We will also review what the dominant integral weights of $GL(X)_0$ are. 
\par
Let $\Lambda(T) = \Hom(T(X)_0, \mathbb{G}_m)$ denote the character lattice, with generators 

\begin{equation}\label{functionals}
    \{\epsilon_1^{1}, \dots, \epsilon_{n_1}^{1}, \epsilon_{1}^2, \dots, \epsilon_{n_2}^{2}, \dots, \epsilon_{1}^{p-1}, \dots, \epsilon_{n_{p-1}}^{p-1}\},
\end{equation}
where $\epsilon_{j}^k$ picks out the $j$-th diagonal entry in $GL(n_k)$ viewed as a subgroup of $GL(X)_0$. For convenience, we will also refer these generators as the set

\[\{\epsilon_1, \dots, \epsilon_{n_1}, \epsilon_{n_1 + 1}, \dots, \epsilon_{n_1 + n_2}, \dots, \epsilon_{n_1 + n_2 + \cdots +n_k}\},\]
and we identify the two via the mapping $\bigoplus_{k=1}^{p-1} L_k^{\oplus n_k}$ to $\bigoplus_{i=1}^{n} X_i$ described above. This can be identified with $\mathbb{Z}^n$ with respect to this ordering. Then, the dominant integral weights of $GL(n_k)$ are given by 

\[\{\lambda = (\lambda_1, \dots, \lambda_{n_k}) \in \Z^{n_k} : \ \lambda_1 \geq \lambda_2 \geq \cdots \geq \lambda_{n_k}\},\]
and the $p$-restricted weights of $GL(n_k)$ are given by 

\[ \{\lambda = (\lambda_1, \dots, \lambda_{n_k}) \in \Z^{n_k} : \ 0 \leq \lambda_i - \lambda_{i+1} < p, \ 1 \leq i \leq n_k - 1\}.\]
The dominant integral weights $\Lambda(T)_+$ for $GL(X)_0$ in the character lattice can be given by the products of sets of dominant integral weights of $GL(n_k)$ for $1 \leq k \leq p-1$ included in $\Lambda(T)$ in the obvious way. One similarly defines the set $\Lambda_1(T)$ of $p$-\textit{restricted weights}. Notice now that the decomposition of the Lie algebra $\gl(X) = \bigoplus_{1 \leq i, j \leq n} X_i \otimes X_j^*$ is a multiplicity-free decomposition into a direct sum of $T(X)$-modules. The weight corresponding to $X_i \otimes X_j^*$ is the $T(X)$-module 

\[(\epsilon_i - \epsilon_j, \un \boxtimes \cdots \boxtimes X_i \boxtimes \cdots \boxtimes X_j^* \boxtimes \cdots \boxtimes \un), \]
where $\un$ denotes the trivial representation for the corresponding group and where $X_i$ and $X_j^*$ are in the $i$-th and $j$-th positions, respectively, and are viewed as the natural representation of $SL(X_i)$ and dual of the natural representation of $SL(X_j)$, respectively.
\par
As a $T_0$-module, it is clear that the weights appearing in this decomposition are 

\[\{\epsilon_i - \epsilon_j\}_{1 \leq i \neq j \leq n}.\]
We will call such weights \textit{roots}, and call a root $\epsilon_i - \epsilon_j$ an \textit{ordinary root} if $X_i \cong X_j$ and a \textit{mixed root} if $X_i \not\cong X_j$. Let $\Phi$ denote the set of all roots, $\Phi_0$ the set of ordinary roots, and $\Phi_{\neq 0}$ the set of mixed roots. We have a set of simple roots given by $\Delta = \{\epsilon_i - \epsilon_{i+1}\}_{1\leq i \leq n-1}$. With respect to this we have our usual notion of positive and negative roots, denoted by $\Phi^+$ and $\Phi^-$, respectively (and extend similar notation for the set of positive or negative roots that are ordinary or mixed). Let $\g_{\hat 0} \coloneqq \mathfrak{t}(X)$, where $\hat 0$ is the zero weight of $T_0$ (such notation is used to not confuse with $\g_0$, the underlying even Lie algebra). We can write $\g_{\epsilon_i - \epsilon_j} = X_i \otimes X_j^*$, and we notice that the bracket restricted to $\g_\alpha \otimes \g_\beta$ lands in $\g_{\alpha + \beta}$ for $\alpha, \beta \in \Phi$ such that $\alpha + \beta$ is a root or zero.  This map is surjective unless $\alpha = -\beta$. 

\begin{rem}
    The root-space decomposition here is just our first pass at such a notion. It is unclear at this time what the right definition of a root system for Lie algebras in $\Ver_p$ should be. What we call ``mixed roots'' are called ``odd'' roots when restricting to the category of supervector spaces but it doesn't seem right to use the term ``odd'' here, especially considering the fact that $\Ver_p$ admits the decomposition $\Ver_p = \Ver_p^+ \boxtimes \sVec_\KK$.
\end{rem}

\section{Frobenius Kernels }\label{sec:frobkernels}

\subsection{The Frobenius twist and the Frobenius morphism}
First let us recall the ordinary Frobenius morphism in the ordinary and super settings. Let $\mathcal{D}$ denote either the category $\Vecc_\KK$ of vector spaces or the category $\sVec_\KK$ of supervector spaces. The \textit{Frobenius twist} is the functor given by sending a vector space $V$ to $V^{(1)}$, where $V^{(1)}$ is the same underlying vector space as $V$ but with the action of $\KK$ twisted by the $p$-th power map $f$ i.e. we can write $V^{(1)} = V \otimes_f \KK$. In this language it is clear that if $A$ is a commutative or Hopf algebra in $\mathcal{D}$ then so is $A^{(1)}$ and that there is a natural homomorphism of algebras $A^{(1)} \rightarrow A$ given by $a \otimes \lambda \mapsto \lambda^p a$. 
\par
Now, if $G$ is an affine group scheme in $\mathcal{D}$, this induces a morphism $G \rightarrow G^{(1)}$, where $G^{(1)}$ is the affine group scheme given by base-changing $G$ with respect to $f$. This map is called \textit{the Frobenius morphism}, and it is a natural transformation of endofunctors on $\mathcal{D}$. If $G$ is defined over the prime subfield of $\KK$, then $G^{(1)}$ can be identified with $G$.
\begin{example}
    For the group scheme $GL(n)$, the Frobenius morphism $F: GL(n) \rightarrow GL(n)$ on each $GL(n)(A)$ for a commutative algebra $A$ over $\KK$ is given by raising each entry of the $A$-valued invertible matrix $GL(n)(A)$ to the $p$-th power.
\end{example}
The Frobenius morphism plays a critical role in understanding the representation theory of affine group schemes in positive characteristic. For instance, it is greatly utilized in showing the finite generation of cohomology of finite group schemes in $\mathcal{D}$ (see \cite{friedlander1997cohomology,drupieski2016cohomological}).
\par
Therefore, it is reasonable that for any symmetric tensor category we want a similar notion. However, the definition here makes use of elements, so it is not suitable for our purposes. Instead, we appeal to a categorical definition of the Frobenius morphism which specializes to the definition above and was introduced in \cite{coulembier2020tannakian} and investigated further in \cite{etingof2020frobenius,coulembier2022frobenius}.
\begin{Def}
    Let $\mathcal{C}$ be a symmetric tensor category over $\KK$. The $r$-th \textit{internal Frobenius twist} of an object $X \in \mathcal{C}$ is the image of the following composition of maps:

    \[\Gamma^{p^r}(X) \xhookrightarrow{} X^{\otimes p^r} \twoheadrightarrow S^{p^r}(X),\]
    where $\Gamma^d$ is the $d$-th divided power functor, $S^d$ is the $d$-th symmetric power functor, and the maps are the obvious inclusions and projections, respectively. We will denote the $r$-th Frobenius twist of $X$ as $X^{(r)}$.
\end{Def}
For our purposes we do not consider the so-called \textit{external Frobenius twist}, so we shall drop the adjective ``internal'' when referring to Frobenius twists. The $r$-th Frobenius twist defines an additive functor $\mathcal{C} \rightarrow \mathcal{C}$; let us denote the first Frobenius twist by $\Fr_+$. In the case $\mathcal{C}$ admits a symmetric tensor functor to $\Ver_p$, we have the following properties (see \cite{coulembier2020tannakian, coulembier2022frobenius}):\footnote{Indeed, such categories are called \textit{Frobenius-exact} as the proposition makes clear.}

\begin{prop}
    Let $\mathcal{C}$ be a symmetric tensor category over $\KK$ that admits a symmetric tensor functor to $\Ver_p$. Then, the following are true:
    \begin{enumerate}
        \item The $r$-th Frobenius twist is given by composing $\Fr_+$ $r$-times. Hence, it makes sense to denote it as $\Fr_+^r$ and write $X^{(r)} = \Fr_+^r(X)$.
        \item The functor $\Fr_+$ is exact.
        \item In the case $\mathcal{C} = \Vecc_\KK$ or $\sVec_\KK$, the Frobenius twist coincides with the usual definition of the Frobenius twist, so our notation is consistent with that in the start of this section.\footnote{Although we assume $\KK$ to be algebraically closed, this statement only requires that $\KK$ be perfect.}
        \item If $\mathcal{C} = \Ver_p$, then given any object $X \in \Ver_p$, the object $\Fr_+(X)$ is isomorphic to the isotypic component of the unit object in $X$. In particular, $\Fr_+$ annihilates the objects $L_2, \dots, L_{p-1}$.
    \end{enumerate}   
\end{prop}
Now, suppose that we have an affine group scheme $G$ in $\Ver_p$ with coordinate ring $\mathcal{O}(G)$ and ordinary subgroup $G_0$. Let $J \subset \mathcal{O}(G)$ be the ideal defining $\mathcal{O}(G_0) = \mathcal{O}(G)/J$, recalling that $J$ is generated by the maximal subobject which is a direct sum of objects $L_i$ with $i \neq 1$. Then, we have the following proposition (see \cite[\S8.1.4]{coulembier2022frobenius}):

\begin{prop}\label{frobhom}
    Let $G$ be an affine group scheme in $\Ver_p$. The Frobenius twist $\Fr_+$ induces a group homomorphism $\Fr_+: G \rightarrow G_0^{(1)}$, the first Frobenius twist of $G_0$ in the ordinary sense. We call this map the \textit{Frobenius morphism}.
\end{prop}

\begin{proof}
    By the functoriality of the Frobenius twist, $\mathcal{O}(G)^{(1)}$ defines an affine group scheme. Then, we consider the following sequence of maps

    \[\mathcal{O}(G)^{(1)} \xhookrightarrow{} S^p(\mathcal{O}(G)) \rightarrow \mathcal{O}(G),\]    
    where the first map is the inclusion and the second map is the $p-$fold multiplication map on $\mathcal{O}(G)^{\otimes p}$, which by commutativity of multiplication factors through $S^{p}(\mathcal{O}(G))$. Now, because $\Fr_+$ is exact and additive and because $S^p(L_i) = 0$ for all simple objects $L_i$ with $i \neq 1$, the image of $J^{(1)}$ under the morphism is zero. Therefore, the map factors through the quotient $\mathcal{O}(G)^{(1)}/J^{(1)}$, which is the coordinate ring of $G_0^{(1)}$. This induces a group homomorphism $G \rightarrow G_0^{(1)}$. 
\end{proof}
This motivates the following definition:

\begin{Def}
    The first \textit{Frobenius kernel} of an affine group scheme $G$ is the kernel of the group homomorphism in Proposition \ref{frobhom}. We will denote it as $G_{(1)}$. Similarly, we can define the $r$-th Frobenius kernel $G_{(r)}$ as the kernel of the group homomorphism induced by $r$-iterates of the Frobenius twist.
\end{Def}
Now, let us deduce some basic properties about Frobenius kernels, most of which are familiar from the ordinary setting. 

\begin{prop}\label{frob_kernel_props}
    Let $G$ be an affine group scheme in a symmetric tensor category $\mathcal{C}$ with symmetric tensor functor to $\Ver_p$. Then, the following are true:

    \begin{enumerate}
        \item $(G_0)_{(r)} = (G_{(r)})_{0}$ and $\Lie(G_{(r)})_{\neq 0} = \Lie(G)_{\neq 0}$ for all $r \geq 1$, which implies $\Lie(G_{(r)}) = \Lie(G)$.
        \item The Frobenius kernel $G_{(r)}$ is a normal subgroup of $G$.  
        \item There exists counit preserving isomorphisms $\mathcal{O}(G_{(r)}) \cong \mathcal{O}((G_0)_{(r)}) \otimes S(\Lie(G)_{\neq 0}^*)$ of left $\mathcal{O}((G_0)_{(r)})$ comodule superalgebras.
        \item The Frobenius kernel $G_{(r)}$ is \textit{infinitesimal}, meaning that the augmentation ideal of $\mathcal{O}(G_{(r)})$ is nilpotent.
    \end{enumerate}
\end{prop}
\begin{proof}
    All four statements are more or less immediate when viewing the Frobenius morphism at the level of Harish-Chandra pairs. See \cite{masuoka2013} for the proofs in the super setting, which carry over to our setting.
\end{proof}

\subsection{Restricted Lie Algebras}
Recall that in the ordinary vector space setting the representation theory of the first Frobenius kernel is captured by the so-called restricted representations of its Lie algebra, which itself admits the structure of a \textit{restricted Lie algebra}. Over supervector spaces, a restricted Lie superalgebra is one whose even part is a restricted Lie algebra and whose odd part is a restricted module over the even part. Such a notion affords a natural generalization to $\Ver_p$, which we will now describe, along with other properties.

\begin{Def}
    A Lie algebra $\mathfrak{g}$ in $\Ver_p$ is a \textit{restricted Lie algebra} if the ordinary part $\g_0$ of $\g$ is a restricted Lie algebra in the usual sense i.e. it has a $p$-mapping $[p]: \g_0 \rightarrow \g_0$ satisfying $\ad(x^{[p]}) = \ad(x)^p$ among other properties (see \cite{demazure1970groupes} for the full definition) 
    and such that $\g_{\neq 0}$ is a restricted $\g_0$-module via the adjoint representation, i.e. 
    
    \[\ad(x)^p - \ad(x^{[p]}): \g_{\neq 0} \rightarrow \g_{\neq 0}\]
    is the zero map for all $x \in \g_0$.
\end{Def}
\begin{prop}
    Let $G$ be an affine group scheme in $\Ver_p$. Then, $\Lie(G)$ is a restricted Lie algebra with respect to the $p$-mapping given by $x \mapsto x^p$ for $x \in \Lie(G)_0$ with multiplication coming from $\Dist(G)$.
\end{prop}
\begin{proof}
    Because $\Lie(G)_0 = \Lie(G_0)$, the ordinary theory tells us that $\Lie(G)_0$ is restricted. Moreover, for each $x \in \Lie(G)_0$, the adjoint action of $x$ on $\Lie(G)_{\neq 0}$ is given by $L_x - R_x$, where $L_x$ and $R_x$ are left and right multiplication, by $x$, respectively. Because multiplication is associative, $L_x$ and $R_x$ commute, the binomial theorem tells us that $(L_x - R_x)^p = L_x^p - R_x^p = L_{x^p} - R_{x^p}$, which means that $\Lie(G)_{\neq 0}$ is a restricted module over $\Lie(G)_0$. 
\end{proof}
Attached to a restricted Lie algebra is its \textit{restricted universal enveloping algebra}, which we will now define:
\begin{Def}
    The \textit{restricted universal enveloping algebra} $U^{[p]}(\g)$ of a restricted Lie algebra $\g$ is the quotient of the usual universal enveloping algebra $U(\g)$ by the ideal generated by the elements $x^{[p]} - x^p$ for $x \in \g_0$.\footnote{In the case $\g = \Lie(G)$ for an affine group scheme $G$, we need to be careful about conflating the multiplication in $\Dist(G)$ with that in $U(\Lie(G))$.}
\end{Def}
\begin{prop}
    Let $G$ be an affine group scheme in $\Ver_p$. Then, the distribution algebra $\Dist(G_{(1)})$ of the first Frobenius kernel of $G$ is isomorphic as a Hopf algebra to the restricted universal enveloping algebra $U^{[p]}(\Lie(G))$.
\end{prop}
\begin{proof}
    By the universal property of $U(\Lie(G_{(1)}))$, we have an algebra homomorphism from $U(\Lie(G_{(1)}))$ to $\Dist(G_{(1)})$. This factors through $U^{[p]}(\Lie(G))$ as $\Lie(G_0)$ is stable under the $p$-th power map by the ordinary theory (seen by identifying it with left-derivations of $\mathcal{O}(G_0)$). Moreover, by the PBW theorem, we have an isomorphism of objects $U^{[p]}(\Lie(G)) \cong U^{[p]}(\Lie(G)_0) \otimes S(\Lie(G)_{\neq 0})$.  Furthermore, as $\Dist((G_0)_{(1)})$-modules, we have an isomorphism $\Dist(G_{(1)}) \cong \Dist((G_0)_{(1)}) \otimes S(\Lie(G)_{\neq 0})$. By the ordinary theory of restricted enveloping algebras, we have an isomorphism of Hopf algebras $U^{[p]}((G_0)_{(1)}) \cong \Dist((G_0)_{(1)})$. Putting this all together, we have a morphism of Hopf algebras between two finite-dimensional algebras of the same size that sends generators to generators, so it is an isomorphism.
\end{proof}
A consequence of this theorem is that the coordinate ring of $G_{(1)}$ is the dual of $U^{[p]}(\Lie(G))$, which implies that the restricted representations of $\Lie(G)$ correspond to representations of $G_{(1)}$, which is what we expected.
\subsection{Representations of Frobenius Kernels of $GL(X)$}\label{sec:frob_reps}
In this subsection we will construct the simple modules over the Frobenius kernels of the general linear group scheme $G \coloneqq GL(X)$ for an object $X \in \Ver_p$ with $X = \bigoplus_{k=1}^{p-1} L_k^{\oplus n_k}$. Let $n = \sum_{k=1}^{p-1} n_k$.
\par
Let us denote by $T$ the standard maximal torus in $GL(X)$, by $B$ the standard Borel subgroup $B(X)$ of $GL(X)$, and by $N^-$ the strictly lower triangular matrices $N^-(X)$ (see \S\ref{sec:glx}). Similarly, let us denote their Lie algebras as $\mathfrak{t}$, $\mathfrak{b}$, and $\mathfrak{n}^-$, respectively. The corresponding Harish-Chandra pairs are $(T_0, \mathfrak{t})$, $(B_0, \mathfrak{b})$, and $(N^-_0, \mathfrak{n}^-)$, respectively. Recall $\Lambda(T)$ is the character lattice of $T_0$.
\par
We will use a standard highest-weight argument via Verma modules to construct the simple $G_{(r)}$-modules. By the PBW theorem and Proposition \ref{frob_kernel_props}, the triangular decomposition on $G$ descends to one on $G_{(r)}$, where the role of $T$, $B$, and $N^-$ are played by $T_{(r)}$, $B_{(r)}$, and $N^-_{(r)}$, respectively.

\begin{prop}\label{prop:frob_kernel_mods}
    The simple $T_{(r)}$-modules, simple $B_{(r)}$-modules, and simple $G_{(r)}$-modules are indexed by the set $\{(\lambda, V)\}$, where 
    
    \[V \in \Irr\left({\bigboxtimes}_{k=1}^{p-1}\Ver_p^+(SL(k))^{\boxtimes n_k}\right)\]  
    and $\lambda \in \Lambda(T)$ is any character. Two such modules $(\lambda, V)$ and $(\mu, W)$ are isomorphic if and only if $\lambda - \mu \in p^r \Lambda(T)$ and $V \cong W$. We will denote the $G_{(r)}$-module corresponding to $(\lambda, V)$ as $L_{(r)}(\lambda, V)$.
\end{prop}
\begin{proof}
    The proof of this proposition is a straightforward generalization of the ordinary theory of representations of Frobenius kernels of algebraic groups. First, by \cite[Corollary 4.2]{venkatesh_glx_2022} we know that $T(X) \cong \prod_{k=1}^{p-1} GL(L_k)^{\times n_k}$, with $GL(L_k) = \mathbb{G}_m \times SL(L_k)$, where $SL(L_k)$ is the finite group scheme with coordinate ring $U(\mathfrak{sl}(L_k))^*$ (this is finite-dimensional by the PBW theorem). It follows that $T_{(r)}$ is isomorphic to the product of group schemes
    
    \[\prod_{k=1}^{p-1}((\mathbb{G}_m)_{(r)} \times SL(L_k))^{\times n_k},\]
    The claim for $T_{(r)}$ then follows from the description of $\Rep SL(L_i)$ from \cite[Corollary 4.8]{venkatesh_glx_2022} and the ordinary theory of representations of Frobenius kernels of $\mathbb{G}_m$. Finally, we can trivially extend from from $T_{(r)}$ to $B_{(r)}$ to get the claim for $B_{(r)}$. Call this module $(\lambda, V)$. Then, we consider the coinduced module $\Dist(G_{(r)}) \otimes_{\Dist(B_{r})} (\lambda, V)$ and quotient out by the maximal submodule, which shows the claim for $G_{(r)}$. Notice that there are no issues with integrability as all group schemes here are finite.
\end{proof}

\section{Proof of The Steinberg Tensor Product Theorem}\label{sec:proof}
We are now ready to prove the Steinberg tensor product theorem for $G = GL(X)$ with $X = \bigoplus_{k=1}^{p-1}L_k^{\oplus n_k} = \bigoplus_{i=1}^n X_k$ as in \ref{sec:glx}. Let us also retain the notations from subsection \ref{sec:glx} and subsection \ref{sec:frob_reps}. In particular recall $T, B, N^-$, $\mathfrak{t}, \mathfrak{b}, \mathfrak{n}^-, \Lambda(T), \Lambda(T)_+, \Lambda_1(T)$ and the notations for the root system $\Phi$.

\begin{lem}\label{lem:factor}
    Let $H$ be an affine group scheme in $\Ver_p$ such that $H_0$ is an ordinary algebraic group defined over the prime subfield of $\KK$. Then, $H = H_{0}H_{(1)}$ and $H_r = (H_{0})_{(r)}H_{(1)}$ for all $r \geq 1$.
\end{lem}
\begin{proof}
    This is the statement of \cite[Lemma 2.6]{binwang_orthosymplectic}, except for supergroups. However, the proof carries through identically to our setting.
\end{proof}
Of course, $G$ satisfies the hypothesis of Lemma \ref{lem:factor} as $G_0$ is just a product of general linear groups. Now, we can proceed as in \S4 of \cite{kujawa2006steinberg}; the following two lemmas are the same as Lemmas 4.1 and 4.2 in \textit{loc. cit.}, respectively, and their proofs carry through almost identically.

\begin{lem}\label{lem:completely_reducible}
    Let $L$ be an irreducible $G$-module. Then $L$ is completely reducible as a $G_{(1)}$-module.
\end{lem}
\begin{proof}
    See \cite[Lemma 4.1]{kujawa2006steinberg}.
\end{proof}
For the simple $T$-module $(\lambda,V)$, recall that the corresponding weight space of $L$ is denoted by $L_{(\lambda, V)}$, where $\lambda \in \Lambda(T)$ and $V \in \Irr\left({\bigboxtimes}_{k=1}^{p-1}\Ver_p^+(SL(k))^{\boxtimes n_k}\right)$. Then, we have the following lemma:
\begin{lem}\label{lem:frob_mod_generator}
    Let $(\lambda, V) \in \Lambda(T)_+ \times \Irr\left({\bigboxtimes}_{k=1}^{p-1}\Ver_p^+(SL(k))^{\boxtimes n_k}\right)$. Then, the $G_{{1}}$-module $\Dist(G_1)L(\lambda, V)_{(\lambda, V)}$ is isomorphic to $L_{(1)}(\lambda, V)$. 
\end{lem}
\begin{proof}
    See \cite[Lemma 4.2]{kujawa2006steinberg}.
\end{proof}
Now, we need to prove the analog to \cite[Lemma 4.3]{kujawa2006steinberg}. We will follow the proof in \textit{loc. cit.} but adapt it for our setting. First, we need some setup. The distribution algebra of the underlying even part is $\Dist(G_0)$ with $\Lie(G)$ isomorphic to the Lie algebra $\bigoplus_{k=1}^{p-1} \mathfrak{gl}(n_k)$. Then, the classical theory tells us that the set 

\[\mathcal{A} \coloneqq \{e_{-\alpha}^{(r)} : \alpha \in \Delta, r \in \Z_{> 0}\}\]
of divided powers $e_{-\alpha}^{(r)} \coloneqq e_{-\alpha}^r/r!$ for the associated root vectors $e_{-\alpha} \in \mathfrak{n}^-_0$ for $\alpha \in \Delta$ and $r \in \Z_{> 0}$ along with $\Dist(B_0)$ generate $\Dist(G_0)$.
\par
Now, recall the set of $p$-restricted weights $\Lambda_1(T)$ of $T_0$ from \S\ref{sec:glx}. The isomorphism classes of simple $G_{(1)}$-modules are in bijection with $\Lambda_1(T) \times \Irr\left({\bigboxtimes}_{k=1}^{p-1}\Ver_p^+(SL(k))^{\boxtimes n_k}\right)$.

\begin{lem}\label{lem:same}
    For $(\lambda, V) \in \Lambda_1(T) \times \Irr\left({\bigboxtimes}_{k=1}^{p-1}\Ver_p^+(SL(k))^{\boxtimes n_k}\right)$, the simple $G$-module $L(\lambda, V)$ is simple upon restriction to $G_{(1)}$ and moreover is $L_{(1)}(\lambda, V)$.
\end{lem}
\begin{proof}
    Let $M = \Dist(G_{(1)})L(\lambda, V)_{(\lambda, V)}$. By Lemma \ref{lem:frob_mod_generator}, $M \cong L_{(1)}(\lambda, V)$, so we just need to show that $M = L(\lambda, V)$, which can be done by showing $M$ is $\Dist(G)$-invariant and therefore all of $L(\lambda, V)$ by irreducibility.
    \par
    Because $\Dist(G)$ is generated by $\Dist(G_{(1)})$ and $\Dist(G_0)$, it suffices to show that $M$ is $\Dist(G_0)$-invariant. Now, note that $\Dist(G_0)$ is generated by $\Dist(B_0)$ and $\mathcal{A}$. The group $G_{(1)}$ is normal in $G$, so $L(\lambda, V)_{(\lambda, V)}$ being a $B_0$-submodule of $L(\lambda, V)$ implies that $M$ is also $B_0$-stable and therefore $\Dist(B_0)$-stable. Therefore, as in the ordinary setting, we just need to show that $M$ is $\mathcal{A}$-stable.
    \par
    Now, by the PBW theorem, $\Dist(G_{(1)})$ is free over $\Dist(B_{(1)})$ and therefore 
    
    \[\Dist(G_{(1)})L(\lambda, V)_{(\lambda, V)} = \Dist(N^-_{(1)})L(\lambda, V)_{(\lambda, V)}.\] 
    Moreover as $\Dist((N^-_0)_{(1)})$-modules $\Dist(N^-_{(1)}) =  \Dist((N^-_0)_{(1)})\otimes S(\mathfrak{n}^-_{\neq 0})$. Now, $\Dist((N^-_0)_{(1)})$ is generated by the divided powers $e_{-\beta}^{(r)}$ of the root vector $e_{-\beta}$ in $\gl(X)_0$ for $\beta \in \Phi_0^+$ and $r \in \{0,1,2,\dots, p-1\}$ and has a basis given 
    
    \[\left\{\prod_{\alpha \in \Phi_0^+} e_{-\beta}^{(r_\beta)}\right\} \]
    for $r_\beta \in \{0,1,2,\dots, p-1\}$ and where the product is with respect to some fixed ordering on $\Phi^+$. Let the degree of such a monomial in this basis be the sum of all the $r_{\beta}$. This endows $\Dist((N^-_0)_{(1)})$ with a grading. Furthermore, $S(\mathfrak{n}^-_{\neq 0})$ has the usual grading. This in totality gives us a grading on $\Dist(N^-_{(1)})$, and we have that $M$ is spanned by subobjects of the form $Y.L(\lambda, V)_{(\lambda, V)}$ where $Y$ in $\Dist(N^-_{(1)})$ is a subobject of degree $d$ and the dot indicates the action of $Y$ on $L(\lambda, V)_{(\lambda, V)}$.
    \par
    Our goal is to now show that the action of $e_{-\alpha}^{(r)} \in \mathcal{A}$  for $r \in \Z_{>0}$ and $\alpha \in \Delta$ on $Y.L(\lambda, V)_{(\lambda, V)}$ will land back in $M$. We will do so by induction on the degree of $Y$, where $Y$ has degree $d$. First of all, suppose the degree of $Y$ is zero. Then, we want to show that $e_{-\alpha}^{(r)}.L(\lambda, V)_{(\lambda, V)}$ is in $M$. This follows from using the representation theory of $SL(2)$ as in the proof of \cite[Lemma 4.3]{kujawa2006steinberg}. So the base case is done.
    \par
    Now, suppose that the degree of $Y$ is $d > 0$. First, let us assume that $Y \subseteq \un \otimes S^d(\mathfrak{n}_{\neq 0}^-)$ under the identification of $\Dist((N^-)_{(1)})$ with  $\Dist((N^-_0)_{(1)}) \otimes S^d(\mathfrak{n}_{\neq 0}^-)$. But then because $S(\mathfrak{n}_{\neq 0}^-)$ is a finite-dimensional $N^-_0$-module, it is clear by the PBW theorem that commuting the divided power $e^{(r)}_{-\alpha}$ past $Y$ will give some other terms in $S^d(\mathfrak{n}_{\neq 0}^-)$ and potentially lower degree terms. Then, we see that $e^{(r)}_{-\alpha}.Y.L(\lambda, V)_{(\lambda, V)}$ is in $M$ by the inductive hypothesis.
    \par
    Finally, the general case of $Y \subseteq \left(\KK\prod_{\beta \in \Phi^+_0}e_{-\beta}^{(r_\beta)}\right) \otimes S^d(\mathfrak{n}_{\neq 0}^-)$ (w.r.t. some fixed ordering on $\Phi_0^+$) follows from combining the previous case with the ordinary setting (again, see \cite{kujawa2006steinberg}). This shows the claim.
\end{proof}
Because $\Fr_+$ gives a group homomorphism from $G$ to $G_0$, any $G_0$-module is also a $G$-module via $\Fr_+$. This gives us a functor $\Fr_{+,\uparrow}: \Rep G_0 \rightarrow \Rep G$. The functor $\Fr_{+,\uparrow}$ has a right-adjoint $\Rep G \rightarrow \Rep G_0$ given by taking $G_{(1)}$-fixed points, which is a consequence of Lemma \ref{lem:factor} (the reasoning given in \cite{kujawa2006steinberg} carries over to our setting). Now, we prove the main theorem of this paper. 

\maintheorem*

\begin{proof}
    The argument in \cite{kujawa2006steinberg} essentially goes through. For 
    
    \[(\lambda, V) \in \Lambda_1(T) \times \Irr\left({\bigboxtimes}_{k=1}^{p-1}\Ver_p^+(SL(k))^{\boxtimes n_k}\right),\]
    we have that $L(\lambda, V)$ is simple as a $G_{(1)}$-module by Lemma \ref{lem:same}. Furthermore, 

    \[H \coloneqq \Hom_{G_{(1)}}(L(\lambda, V), L(\lambda + p\mu, V)) \neq 0\]
    by Lemma \ref{lem:frob_mod_generator} and Proposition \ref{prop:frob_kernel_mods}. Now, $H$ is a $G$-module because $G$ acts on $G_{(1)}$ by conjugation. We can define a $G$-module homomorphism $H \otimes L(\lambda, V) \rightarrow L(\lambda + p\mu, V)$ given by evaluation. This map is surjective because $H \neq 0$ and $L(\lambda + p\mu, V)$ is irreducible. Now, for an object in $\Ver_p$, let the function $\ell: \Ver_p \rightarrow \Z_{\geq 0}$ denote its length (for an ordinary vector space, this is just its dimension). Computing lengths, we have

    \begin{align*}
        \ell(H\otimes L(\lambda, V)) &= \dim_\KK H \cdot \ell(L(\lambda, V)) \\
        &\leq \frac{\ell(L(\lambda + p\mu,V))}{\ell(L(\lambda, V))}\ell(L(\lambda, V)) \\
        &= \ell(L(\lambda + p\mu, V)).
    \end{align*}
    The first line follows because $H$ is an ordinary vector space. The second line follows from combining Schur's Lemma with the facts that $L(\lambda, V)$ is a simple $G_{(1)}$-module and $L(\lambda + p\mu, V)$ is completely reducible as a $G_{(1)}$-module by Lemma \ref{lem:completely_reducible}. In other words, the hom-space multiplicity is how many times $L(\lambda, V)$ appears in a decomposition of $L(\lambda + p\mu, V)$ into a direct sum of simples, which is obviously bounded by ratio of lengths. We deduce the map is an isomorphism as it is a surjection between two objects of the same length.

    Since $G_{(1)}$ acts trivially on $H$, it must be the inflation of some $G_0$-module $M$ via $\Fr_{+,\uparrow}$ as taking $G_{(1)}$-fixed points is right-adjoint to $\Fr_{+,\uparrow}$. This module $M$ is irreducible because $L(\lambda + p\mu, V)$ is irreducible. Finally, we have that
    
    \begin{align*}
        H = \Hom_{G_{(1)}}(L(\lambda, V), L(\lambda + p\mu, V)) &= \Hom_{\mathcal{C}}(\un, \underline{\mathrm{Hom}}_{G_{(1)}}(L(\lambda, V), L(\lambda + p\mu, V))) \\
        &= \Hom_{\mathcal{C}}(\un, L(\lambda+p\mu, V) \otimes L(\lambda, V)^*),
    \end{align*}
    which as a $G_0$-module has highest weight $p\mu$. Here, $\underline{\mathrm{Hom}}_{G_{(1)}}$ denotes the internal hom of $\Rep G_{(1)}$, and the $*$ in the superscript denotes the dual of a module. It follows that $M \cong L_0(\mu)$, and we are done.

\end{proof}

We expect analogs of this theorem to be true for orthogonal, symplectic, queer, and periplectic groups in $\Ver_p$, but we have not classified their representations yet. In fact, we should be able to prove this theorem for a large class of affine group schemes in $\Ver_p$.  Let us call a representation of an affine group scheme $G$ \textit{one-irreducible} if its restriction to $G_{(1)}$ is irreducible. We state the following conjecture:

\begin{conj}
Let $G$ be an affine group scheme with reductive even subgroup $G_0$. Any irreducible representation of $G$ is uniquely a tensor product of a one-irreducible representation and one inflated from $G_0$.
\end{conj}
If this is not true, what further conditions need to be imposed on $G$ to make it true?
\printbibliography

\end{document}